\documentclass[12pt]{article}
\pdfoutput=1
\usepackage{graphicx}
\usepackage{mathptmx}
\usepackage{mathrsfs}
\usepackage{amssymb}
\usepackage{amsmath}
\usepackage{graphicx}
\usepackage{amsmath,amsfonts}
\usepackage{algorithm}
\usepackage{algorithmic}
\usepackage[round]{natbib}
\usepackage{amscd}
\usepackage[mathscr]{eucal}

\newcommand{\uN}{\mathbb{N}}
\newcommand{\uA}{\mathcal{A}}
\newcommand{\uD}{\mathcal{D}}

\newcommand{\opP}{\mathbf{P}}
\newcommand{\opQ}{\mathbf{Q}}

\newtheorem{thm}{Theorem}[section]
\newtheorem{exam}{Example}[section]

\newtheorem{rem}{Remark}[section]
\newtheorem{defn}{Definition}[section]
\newtheorem{cor}{Corollary}[section]

\newenvironment{proof}{\noindent\\ \noindent\relax{\sc
     Proof}}{{\samepage\par\nopagebreak\hbox
     to\hsize{\hfill$\Box$}}}
\newcommand{\be}{\begin{equation}} \newcommand{\ee}{\end{equation}}
\newcommand{\bd}{\begin{displaymath}} \newcommand{\ed}{\end{displaymath}}
\newcommand{\ba}{\begin{align}} \newcommand{\ea}{\end{align}}
\newcommand{\baa}{\begin{align*}} \newcommand{\eaa}{\end{align*}}
\newcommand{\ben}{\begin{enumerate}} \newcommand{\een}{\end{enumerate}}
\newcommand{\bi}{\begin{itemize}} \newcommand{\ei}{\end{itemize}}

\newcommand{\ud}{\mathrm{d}}
\newcommand{\E}[1]{\operatorname{E}\left[ #1 \right]}

\newcommand{\supp}[1]{\operatorname{supp}\left( #1 \right)}

\begin{document}

\title{Prevalence problem in the set of quadratic stochastic operators acting on
$L^{1}$}
\author{{\sc Krzysztof Bartoszek} and {\sc Ma\l gorzata Pu\l ka}} 

\maketitle

\begin{abstract}
This paper is devoted to the study of the problem of prevalence in the class of quadratic stochastic operators
acting on the $L^{1}$ space for the uniform topology. We obtain that the set of 
norm quasi--mixing quadratic stochastic
operators is a dense and open set in the topology induced by a very natural metric. This shows the typical long--term
behaviour of iterates of quadratic stochastic operators.
\end{abstract}


\section{Introduction}
\label{intro}
Iterates of Markov operators have been studied for a long time due to their wide range of
applications in different areas of science and technology.
Today, motivated by numerous biological and physical phenomena, there is a growing interest in nonlinear methods.
Recently \citet{VKol2010} wrote a detailed overview of the theory of nonlinear Markov processes.
Here we study the so--called quadratic stochastic operators which are bilinear by definition.
They were first introduced by \citet{SBer1924} to describe the evolution of a discrete probability
distribution of a finite number of biotypes in a process of inheritance. Since then
the field has been steadily evolving and \citet{Ke1,L} provide a good overview of it.
Furthermore \citet{RGanFMukURoz2011} discuss a number of open problems in it.

A typical question when working with quadratic stochastic operators is the long--term behaviour of their iterates.
\citet{WBarMPul2013}
introduced and studied different types of asymptotic behaviours of quadratic
stochastic operators acting on the $\ell^{1}$ space. These results were subsequently
extended to the $L^{1}$ case by us \citep{KBarMPul2015}.
In this paper we are especially interested in answering the question what a typical
(generic with respect to a specified metric topology) quadratic stochastic operator acting on $L^{1}$
looks like, i.e. we search for prevalent subsets in the class of quadratic stochastic operators.
However, our work is more than a continuation of our previous results \citep{KBarMPul2015}.
In our previous paper \citep{KBarMPul2015} we showed equivalent conditions for asymptotic
stability of quadratic stochastic operators of a very special type, namely kernel ones. 
Here we do not restrict ourselves to this class but study the geometry of the set
of all quadratic stochastic operators. 

\section{Basic concepts}
\label{sec:2}
In this section we introduce notation alongside some basic definitions and properties.
Let $\left(X, \uA, \mu\right)$ be a separable $\sigma$--finite
measure space. Throughout the paper by $L^{1}$ we denote the (separable)
Banach lattice of real and $\uA$-measurable functions $f$ such that
$|f|$ is $\mu$-integrable and equipped with the norm $\Vert f \Vert_1 :=
\int_X |f| d\mu$. By $\uD := \uD\left(X, \uA,\mu\right)$ we denote the convex set of all
\emph{densities} on $X$, i.e.
\begin{displaymath}
\uD = \left\{ f \in L^1: f \geq 0,
\left\Vert f\right\Vert_{1} = 1\right\}.
\end{displaymath}
We say that a linear operator $P \colon L^1 \to
L^1$ is \emph{Markov} (or \emph{stochastic}) if
\begin{displaymath}
P f \geq 0 \quad \textrm{ and } \quad \left\Vert P f\right\Vert_1 =
\left\Vert f\right\Vert_1
\end{displaymath}
for all $f \geq 0$, $f \in L^1$. Clearly
$\left\Vert \left| P\right| \right\Vert := \sup_{\Vert f\Vert_1 = 1} \Vert Pf \Vert_1 =1$ and $P\left(\uD\right) \subset
\uD$. The sequence of such operators denoted by $\opP = (
P^{[n,n+1]})_{n \geq 0}$ is called a \emph{(discrete time)
nonhomogeneous chain of stochastic operators} or shorter, a \emph{nonhomogeneous Markov chain}.
For $m,n \in \uN_0 := \uN \cup \{ 0 \}$, $n-m \geq 1$, and any $f \in L^1$ we
set 
\bd
P^{[m,n]}f = P^{[n-1,n]}( P^{[n-2,n-1]}(
\cdots ( P^{[m,m+1]} f) \cdots )).\ed
We denote by $I$ the identity operator
and naturally $P^{[n,n]}=I$. If for
all $n \geq 0$ one has $P^{[n,n+1]} = P$, then we say
that $\opP = (P)_{n \geq 0}$ is \emph{homogeneous}. The set of all chains of Markov
operators $\opP = ( P^{[n,n+1]})_{n \geq 0}$ will be
denoted by $\mathfrak{S}$.

Different types of asymptotic behaviours as well as residualities in the set $\mathfrak{S}$ (endowed
with suitable natural metric topology) have been recently intensively
studied by \citet{MPul2011,MPul2012}. 
Following \citet{MPul2012} we introduce the below types of asymptotic stabilities.

\begin{defn}
A discrete time nonhomogeneous chain of stochastic operators $\opP \in \mathfrak{S}$ is
called
\begin{enumerate}
\item uniformly asymptotically stable if there exists a unique
$f_*\in \mathcal{D} $ such that for any $m \geq 0$

$$\lim\limits_{n\to\infty}\sup\limits_{f\in \mathcal{D}} \left\Vert P^{[m,n]}f - f_* \right\Vert_1 = 0 ,$$
\item  almost uniformly asymptotically stable if for any $m \geq 0$

$$\lim\limits_{n\to\infty}\sup\limits_{f, g\in \mathcal{D}} \left\Vert P^{[m,n]}f - P^{[m,n]}g \right\Vert_1 = 0 , $$
\item  strong asymptotically stable if there exists a unique $f_*\in \mathcal{D}$
such that for all $m \geq 0$ and $f \in \mathcal{D} $

$$\lim\limits_{n\to\infty} \left\Vert P^{[m,n]}f - f_* \right\Vert_1 = 0 ,$$
\item  strong almost asymptotically stable if for all $m \geq 0$ and $f, g \in \mathcal{D}$

$$\lim\limits_{n\to\infty} \left\Vert P^{[m,n]}f - P^{[m,n]}g \right\Vert_1 = 0 \ .$$
\end{enumerate}
\end{defn}

The long--term behaviour of iterates of nonhomogeneous stochastic operators
is still a subject of interest, despite having been intensively studied.
\citet{UHer1988} provides a
comprehensive review of different asymptotic
behaviours of nonhomogeneous Markov chains.
Very recently \citet{FMuk2013a,FMuk2013b,FMuk2013c} contributed further this direction.

One of the reasons for the growing interest in quadratic stochastic operators is their
applicability to biological sciences, especially in the modelling of reproducing populations
\citep{KBarMPul2013,NGanJDaoMUsm2010,NGanMSabUJam2013,NGanMSabANaw2014,NGanUJam2012,RRudPZwo2015,PZwo2015}
and describing multi--agent systems \citep{MSabKSab2014a,MSabKSab2014b}. This interest has
resulted in numerous recent works on general \citep{VAle2015,FMukNSup2015,FMukMSabIQar2013} and ergodic 
\citep{NGanDZan2012,NGanHAkiFMuk2006,NGanDZan2004,MPul2011,MSab2007} properties of 
QSOs, specific subclasses of QSOs like Volterra \citep{FMuk2014,FMukMSab2010,MSab2012a,MSab2013} or other 
\citep{NGanNHam2014}.
These new results build on a rich previous literature \citep[e.g.][]{RGan1989,RGan1993,RGanDESh2006,TSarNGan1990,MZak1978}.
We now define a quadratic stochastic operator acting on the $L^{1}$
space \citep[cf.][]{KBarMPul2015}.

\begin{defn}\label{dfQSO}
A bilinear operator $\opQ \colon L^1 \times L^1 \to L^1$ is called a
quadratic stochastic operator if

$$ \opQ(f,g) \geq 0\ ,  \quad
\opQ(f,g) = \opQ(g,f) \quad \textrm{ and } \quad \Vert\opQ(f,g) \Vert_1 = \Vert f\Vert_1 \Vert g\Vert_1  $$
for all $f,g \geq 0$, $f,g \in L^1$.
\end{defn}

Notice that $\opQ$ is bounded as $\sup_{\Vert f\Vert_1 = 1, \Vert g\Vert_1 = 1} \Vert \opQ(f,g)\Vert_1 = 1$.
Moreover if $\widetilde{f} \ge f \ge 0$ and $\widetilde{g} \ge g \ge 0$
then $\opQ(\widetilde{f},\widetilde{g}) \ge \opQ(f,g)$.
Clearly, $\opQ ( \uD \times \uD ) \subseteq\uD$.
The family of all quadratic stochastic
operators will be denoted by $\mathfrak{Q}$.

In population genetics
special attention is paid to a nonlinear mapping
$\mathcal{D} \ni f \mapsto \mathbb{Q}(f) := \mathbf{Q}(f,f) \in \uD$.
The iterates $\mathbb{Q}^{n}(f)$, where $n = 0,1,2,\ldots$,
model the evolution of a distribution of some (continuous) trait of an inbreeding or
hermaphroditic population. We \citep{KBarMPul2013} previously discussed in detail biological
interpretations of different types of asymptotic behaviours of quadratic stochastic operators in
the discrete $\ell^{1} = L^{1}(\mathbb{N}, 2^{\mathbb{N}}, \textrm{ counting measure})$ case.
Clearly, $\mathbb{Q}( \uD) \subseteq \uD$.

We \citep{KBarMPul2015} showed that $\opQ$ is continuous on $L^1 \times L^1$
and uniformly continuous if applied to vectors from the unit ball in $L^1$. In particular,
$\mathbb{Q}$ is uniformly continuous on the unit ball in $L^1$.

We will endow the set $\mathfrak{Q}$ with a metric structure relevant to the uniform convergence on $\mathcal{D}$.
Given $\mathbf{Q}_1, \mathbf{Q}_2 \in \mathfrak{Q}$ let us define
\begin{enumerate}
\item[(1)]  $d_u(\opQ_1,\opQ_2) = \sup_{f \in
\uD} \left\Vert\mathbb{Q}_1(f) - \mathbb{Q}_2(f) \right\Vert_1$,
\item[(2)] $\widehat{d}_u(\opQ_1,\opQ_2) = \sup_{f,g \in
\uD} \left\Vert\mathbf{Q}_1(f,g) - \mathbf{Q}_2(f,g) \right\Vert_1$.
\end{enumerate}

\noindent Notice that the above metrics are equivalent. Clearly we have

$$ d_u (\opQ_1, \opQ_2) = \sup\limits_{f \in \uD} \Vert\mathbb{Q}_1 (f) - \mathbb{Q}_2 (f)\Vert_1 \leq \sup\limits_{f,g \in \uD} \Vert \opQ_1(f,g) - \opQ_2(f,g)\Vert_1 = \widehat{d}_u(\opQ_1, \opQ_2) \ . $$
On the other hand, let 

$$\Vert\opQ_1(\widetilde{f},\widetilde{g}) - \opQ_2(\widetilde{f},\widetilde{g})\Vert_1 \geq (1-\varepsilon)\widehat{d}_u(\opQ_1,\opQ_2) $$
for some $\widetilde{f},\widetilde{g} \in \uD$ (where later on we let $\varepsilon \to 0^{+}$).
If $\widetilde{f} = \widetilde{g}$ $\mu$ - a.e. then

$$ d_u (\opQ_1, \opQ_2)  \geq \Vert\mathbb{Q}_1(\widetilde{f}) - \mathbb{Q}_2(\widetilde{f})\Vert_1 \geq \frac{1}{2}\widehat{d}_u(\opQ_1, \opQ_2) \geq \frac{1}{4}\widehat{d}_u(\opQ_1, \opQ_2)\ .$$
If $d_u(\opQ_1, \opQ_2)  \leq \tfrac{1}{2}\widehat{d}_u(\opQ_1, \opQ_2) $ then by bilinearity of $\opQ_1$
and $\opQ_2$ for $h := \tfrac{1}{2}\widetilde{f} + \tfrac{1}{2}\widetilde{g} \in \uD$ we have

$$
\begin{aligned}
d_u(\opQ_1, \opQ_2) &\geq \left\Vert\mathbb{Q}_1(h) - \mathbb{Q}_2 (h)\right\Vert_1 \\
&= \left\Vert\mathbb{Q}_1(\frac{1}{2} \widetilde{f}+ \frac{1}{2}\widetilde{g}) - \mathbb{Q}_2(\frac{1}{2}\widetilde{f} + \frac{1}{2}\widetilde{g})\right\Vert_1 \\
&\geq \frac{1}{2}\left\Vert\mathbf{Q}_1(\widetilde{f},\widetilde{g}) - \mathbf{Q}_2(\widetilde{f},\widetilde{g})\right\Vert_1 - \frac{1}{4}\left\Vert\mathbb{Q}_1(\widetilde{f}) - \mathbb{Q}_2(\widetilde{f})\right\Vert_1 - \frac{1}{4} \left\Vert\mathbb{Q}_1(\widetilde{g})  - \mathbb{Q}_2(\widetilde{g})\right\Vert_1 \\
&\geq \frac{1}{2}(1-\varepsilon)\widehat{d}_u(\opQ_1, \opQ_2) - \frac{1}{8}\widehat{d}_u(\opQ_1, \opQ_2) - \frac{1}{8}\widehat{d}_u(\opQ_1, \opQ_2)\\
&= \left(\frac{1}{2} - \frac{\varepsilon}{2} - \frac{1}{4}\right)\widehat{d}_u(\opQ_1, \opQ_2) \ .
\end{aligned}
$$
Letting $\varepsilon \to 0^+$ finally we obtain

$$
\frac{1}{4}\widehat{d}_u(\opQ_1, \opQ_2) \leq  d_u(\opQ_1, \opQ_2) \leq \widehat{d}_u(\opQ_1, \opQ_2) \ .$$

A natural question arises concerning the necessity of the symmetry condition in the definition of a quadratic
stochastic operator. Indeed, in general one could consider a
\emph{nonsymmetric quadratic stochastic operator}, i.e. a bilinear
operator $\opQ \colon L^1 \times L^1 \to L^1$ such that
$ \opQ(f,g) \geq 0$ and $\Vert\opQ(f,g) \Vert_1 = \Vert f\Vert_1 \Vert g\Vert_1  $ for all $f,g \geq 0$,
$f,g \in L^1$. However this situation is a topic for further research.
A small number of fragments of \citet{WBarMPul2013}'s and our \citep{KBarMPul2015}'s proofs do not follow through in this case.
Furthermore we notice that $d_u$ will not be a metric.
This is as for $\opQ^\flat, \opQ^\sharp \in \mathfrak{Q}$
defined by $\opQ^\flat(f,g) = f (\int_X g d\mu)$ and $\opQ^\sharp(f,g) = g (\int_X f d\mu)$
for any $f,g \in L^{1}$ we have that both $\opQ^\flat$ and $\opQ^\sharp$ are nonsymmetric quadratic stochastic operators such that
$\opQ^\flat \neq \opQ^\sharp$ and
$d_u(\opQ^\flat, \opQ^\sharp) = \sup_{f \in \uD} \Vert\mathbb{Q}^\flat(f) - \mathbb{Q}^\sharp(f)\Vert_1 = 0$.
On the other hand, it is easy to check that $\widehat{d}_u$ is
a metric in the class of nonsymmetric quadratic stochastic operators.

Different types of asymptotic behaviours of quadratic stochastic operators and the relationships between
them have been recently intensively studied by \citet{WBarMPul2013}, \citet{KBarMPul2015} and \citet{RRudPZwo2015}.
We follow \citet{KBarMPul2015} in defining a quadratic stochastic operator.

\begin{defn}
A quadratic stochastic operator  $\mathbf{Q}\in \mathfrak{Q}$ is called:
\begin{enumerate}
\item norm mixing (uniformly asymptotically stable) if there exists a density
$f\in \mathcal{D} $ such that

$$\lim\limits_{n\to\infty}\sup\limits_{g\in
\mathcal{D}} \left\Vert \mathbb{Q}^n(g) - f \right\Vert_1 = 0 \ , $$
\item  strong mixing (asymptotically stable) if there exists a density
$f\in \mathcal{D}$ such that for all $g \in \mathcal{D} $ we have

$$\lim\limits_{n\to\infty} \left\Vert \mathbb{Q}^n(g) - f
\right\Vert_1 = 0 \  , $$
\item  strong almost mixing if  for all  $g,h \in \mathcal{D} $ we have

$$\lim\limits_{n\to\infty} \left\Vert \mathbb{Q}^n(g) - \mathbb{Q}^n(h) \right\Vert_1 = 0 \ .$$
\end{enumerate}
\end{defn}

The sets of all norm mixing, strong mixing, strong almost mixing quadratic stochastic operators are denoted
respectively by $\mathfrak{Q}_{nm}, $  $\mathfrak{Q}_{sm}, $ $\mathfrak{Q}_{sam}. $
It can be easily seen \citep[cf.][]{KBarMPul2015} that
$ \mathfrak{Q}_{nm} \subsetneq \mathfrak{Q}_{sm} \subsetneq \mathfrak{Q}_{sam}$.

We now introduce the relation between quadratic stochastic operators
and (linear) Markov operators. \citet{NGan1991} was the first to introduce this approach and
it was used recently by \citet{WBarMPul2013} and us \citep{KBarMPul2015}.
This correspondence allows one to study a linear model instead of a nonlinear one.
Again we follow \citet{KBarMPul2015}.

\begin{defn}
Let $\opQ \in \mathfrak{Q}$. For arbitrarily fixed initial density function $g \in \uD$ a
nonhomogeneous Markov chain associated with $\opQ$ and $g \in \uD$ is defined as
a sequence $\opP_{g} = ( P_{g}^{[n,n+1]})_{n \geq 0}$  of Markov operators $P_{g}^{[n,n+1]} \colon L^1 \to L^1$ of the form

$$P_{g}^{[n,n+1]}(h) := \opQ(\mathbb{Q}^{n}(g), h) \ .$$
\end{defn}

Notice that if the initial density $f$ is $\mathbb{Q}$--invariant (i.e. $\mathbb{Q}(f) = f$),
then the associated Markov chain $\opP_{f}$ is homogeneous as for any $h \in L^1$ the expression
$\opQ(\mathbb{Q}^n(f),h) = \opQ(f,h)$ does not depend on $n$.
In this case we abbreviate the notation and write $P^{[n,n+1]}_f =: P_{f}$ and $P^{[0,n]}_f =: P_{f}^{n}$.

Norm mixing of a quadratic stochastic operator $\opQ \in \mathfrak{Q}$ is evidently correlated with asymptotic behaviour of
its associated nonhomogeneous Markov chain as we proved \citep{KBarMPul2015}. Namely we
have

\begin{thm}[\citet{KBarMPul2015}]\label{thm6}
Let $\opQ$ be a
quadratic stochastic operator.
The following conditions are equivalent:
\begin{enumerate}
\item[(1)] \ There exists $f \in \mathcal{D}$ such that

$$\lim\limits_{n \to \infty} \sup\limits_{g \in \mathcal{D} } \left\Vert \mathbb{Q}^{n}(g) - f \right\Vert_{1} = 0. $$
\item[(2)] \ There exists $f \in \mathcal{D}$ such that

$$\lim\limits_{n \to \infty} \sup\limits_{g,h \in \mathcal{D} }\left\Vert P_{h}^{[0,n]}(g) - f \right\Vert_{1} = 0. $$
\item[(3)] \ There exists $f \in \mathcal{D}$ such that for every $m \ge 0$ we have

$$\lim\limits_{n \to \infty} \sup\limits_{g,h \in \mathcal{D} }\left\Vert P_{h}^{[m,n]}(g) - f \right\Vert_{1} = 0, $$
i.e. independently of the seed $g \in \mathcal{D}$, all nonhomogeneous Markov chains
$\mathbf{P}_{g} = (P_{g}^{[n,n+1]})_{n \ge 0} $
are norm mixing with a common limit distribution $f$ and the rate of
convergence is uniform for $g$.
\end{enumerate}
\end{thm}

\section{Mutual correspondence between $L^1$ and $\ell^1$ spaces and its consequences}
\label{secConnection}
Let us recall that a measurable countable partition $\xi := \{ B_{k}\}_{k=1}^{\infty}$  of $X$  is called
consistent with $\sigma$--finite
measure $\mu$ if $0 < \mu(B_{k}) < \infty$ for all $k$. Such partitions exist since the measure $\mu$ is $\sigma$--finite.
Given a consistent measurable countable partition $\xi := \{ B_{k}\}$ and any $f_{1}, f_{2} \in L^1$ we write

$$
f_{1} \sim f_{2} \Leftrightarrow \forall_{k} \int_{B_{k}} f_{1} \ud \mu=  \int_{B_{k}} f_{2} \ud \mu \
$$
what defines an equivalence relationship on the $L^{1}$ space.
Hence each equivalence class (taking $f \in L^{1}$ as its representative)
can be associated with an element $p_{f} \in \ell^{1}$, namely take

$$\ell^{1} \ni p_{f} = \left(\int_{B_{1}} f \ud \mu, \int_{B_{2}} f \ud \mu, \dots  \right).$$
Notice that the coordinates of the vector $p_{f}$ are actually the conditional expectations
$\E{\cdot \vert B_{k}}$ for the density $f$
and measure $\mu$.

Motivated by the mutual correspondence between $L^1$ and $\ell^1$ we recall the definition of a quadratic
stochastic operator on $\ell^1$ \citep{WBarMPul2013}.
\begin{defn}
A quadratic stochastic operator is defined as a cubic array of nonnegative real numbers
$\mathbf{Q}_{seq}=  [q_{ij,k}]_{i,j,k \geq 1}$ if it satisfies
\begin{enumerate}
\item[(D1)]  $ 0 \leq q_{ij,k} = q_{ji,k}\leq 1$ for all $i,j,k \geq 1 $,
\item[(D2)]  $\sum_{k=1}q_{ij, k} = 1$ for any pair $(i,j)$.
\end{enumerate}
\end{defn}
Such a cubic matrix $\mathbf{Q}_{seq}$ may be viewed  as a bilinear mapping
$\mathbf{Q}_{seq}: \ell^1 \times \ell^1 \to \ell^1 $ if we
set $\mathbf{Q}_{seq}((x_1,x_2,\ldots), (y_1,y_2,\ldots))_k = \sum_{i, j=1}x_iy_jq_{ij,k}$ for any $k \geq 1$.

We denote by $\mathfrak{Q}:= \mathfrak{Q}(L^{1})$ and $\mathfrak{Q}(\ell^{1})$ the sets of all
quadratic stochastic operators defined on $L^{1}$ and $\ell^{1}$ respectively.
Let $\mathbb{E} \colon \mathfrak{Q}(L^{1}) \to \mathfrak{Q}(\ell^{1})$ be defined by

$$
\mathbb{E}\opQ(\underline{x},\underline{y}) = 
\sum\limits_{i=1}^{\infty} \left(\int\limits_{B_i} 
\opQ(\sum\limits_{k=1}^{\infty}\frac{x_{k}}{\mu(B_{k})}\mathbf{1}_{B_{k}},
\sum\limits_{l=1}^{\infty}\frac{y_{l}}{\mu(B_{l})}\mathbf{1}_{B_{l}})  d\mu \right) \underline{e_i},
$$
where $\underline{x}=(x_1,x_2,\ldots), \underline{y}=(y_1,y_2,\ldots)\in \ell^{1}$, $\underline{e_i} := (\delta_{ji})_{j \geq 1}$ and $\delta_{ji} = 1$ for $j=i$ and $\delta_{ji} = 0$ for $j \neq i$.
Let us notice that the mapping $\mathbb{E}$ is continuous. Indeed,
let $\opQ_1 \in \mathfrak{Q}(L^1)$
and $\varepsilon >0$ be fixed and choose $\delta = \varepsilon$. Then for any $\opQ_2 \in \mathfrak{Q}(L^1)$
satisfying $d_u(\opQ_1, \opQ_2) = \sup_{f \in \uD} \Vert\mathbb{Q}_1(f) - \mathbb{Q}_2(f)\Vert_1 < \delta$
we have

$$
\begin{array}{rcl}
\sup\limits_{\{\underline{x} \in \ell^{1}: \Vert \underline{x}\Vert_{\ell^{1}}=1, x_{k}\ge 0\}} 
\left\Vert\mathbb{E}\opQ_1(\underline{x},\underline{x}) - \mathbb{E}\opQ_2(\underline{x},\underline{x})\right\Vert_{\ell^1}
&\leq& \sup\limits_{f \in \uD} \sum\limits_{k=1}^\infty \left| \int_{B_k} \mathbb{Q}_1(f)d\mu- \int_{B_k} \mathbb{Q}_2(f)d\mu \right|
\\ &\leq& \sup\limits_{f \in \uD} \sum\limits_{k=1}^\infty \int_{B_k} \left| \mathbb{Q}_1(f) - \mathbb{Q}_2(f) \right| d\mu
\\ &=& \sup\limits_{f \in \uD} \int_{X} \left| \mathbb{Q}_1(f) - \mathbb{Q}_2(f) \right| d\mu
\\ &=& \sup\limits_{f \in \uD} \left\Vert \mathbb{Q}_1(f) - \mathbb{Q}_2(f) \right\Vert_1 < \varepsilon.
\end{array}
$$

\begin{rem}\label{Rem_cont}
By the continuity of the mapping $\mathbb{E}$ we obtain that if $\mathcal{Q}$ is an
open subset of $\mathfrak{Q}(\ell^1)$ then its preimage $\mathbb{E}^{-1}(\mathcal{Q})$ is an open subset in
$\mathfrak{Q}(L^1)$. We will use this fact to describe the geometric structure of the
set $\mathfrak{Q}(L^1)$.
\end{rem}

\section{Prevalence problem in the set of quadratic stochastic operators}
\label{sec:4}
In this section we study whether there is any typical asymptotic behaviour of quadratic stochastic operators.
We use symbols $\textrm{Int}(A)$, $A^{\complement}$ and $\textrm{diam}(A)$ to denote the interior,
the complement and the diameter of the set $A$.
Of course the whole below discussion excludes the trivial cases of $\dim \ell^{1} = 1$ 
and $\dim L^{1}=1$.

We begin our study with a description of the class of (uniformly) asymptotically stable quadratic stochastic operators.
Recall that \citet{WBarMPul2013} proved that the interior of the set $\mathfrak{Q}_{sm}^{\complement}(\ell^1)$,
quadratic stochastic operators acting on $\ell^{1}$ which are not strong mixing, is nonempty.
Taking into account the mutual correspondence between the spaces $L^{1}$ and $\ell^{1}$ discussed in
the previous section as well as Remark \ref{Rem_cont} we obtain the following

\begin{cor}
The set \ $\mathbb{E}^{-1}(\textrm{Int}(\mathfrak{Q}_{sm}^{\complement}(\ell^1)))$ \ is nonempty, open and
moreover
\bd\mathbb{E}^{-1}(\textrm{Int}(\mathfrak{Q}_{sm}^{\complement}(\ell^1))) \subseteq \mathfrak{Q}_{sm}^{\complement}.\ed
In particular, $\mathfrak{Q}_{nm}^{\complement}$ contains a nonempty open set.
\end{cor}

\noindent On the other hand, \citet{WBarMPul2013} also showed that the interior of
the set $\mathfrak{Q}_{nm}^{\complement}(\ell^1)$, quadratic stochastic operators acting on $\ell^{1}$
which are norm mixing, is nonempty as well. We will show directly that $\textrm{Int}(\mathfrak{Q}_{nm}(L^{1}))$
is non--empty.

\begin{exam}\label{exQnm}{\sl
Let us define
$\mathbf{Q}^{\diamond}\in \mathfrak{Q}$ by
\mbox{
$\mathbf{Q}^{\diamond}(f,g) = (\int_X f d\mu) (\int_X g d\mu) h$,} 
for any $f, g \in L^{1}$ and a fixed density function
$h \in \uD$. Clearly $\mathbf{Q}^{\diamond}$ is norm mixing.
Suppose that $\mathbf{Q}\in \mathfrak{Q}$ satisfies
$\widehat{d}_u(\mathbf{Q}, \mathbf{Q}^{\diamond} ) = \varepsilon < \frac{1}{3}$.
We will show that such a $\mathbf{Q}$ is also norm mixing.
For any $f,g,u,v \in \uD$ we have

$$
\begin{aligned}
\Vert\mathbf{Q}(f,g) - \mathbf{Q}(u,v) \Vert_{1} \leq \Vert\mathbf{Q}(f,g) - \mathbf{Q}^{\diamond}(f,g)\Vert_{1} +
\Vert\mathbf{Q}(u,v) - \mathbf{Q}^{\diamond}(u,v)  \Vert_{1} \leq 2\varepsilon \ .
\end{aligned}
$$
Thus, $\mathrm{diam}(\overline{\mathbb{Q}(\uD)})\le 2\varepsilon$. Let $u,v \in \overline{\mathbb{Q}(\uD)}$
and so $\Vert u-v \Vert_{1} = \kappa \le 2\varepsilon$. 
We denote by $\wedge$ the ordinary minimum in $L^{1}$
and let $f=u - u \wedge v$, $g=v- u \wedge v$. 
Since $L^{1}$ is an AL--space, then $\Vert u \wedge v \Vert_{1}=1-\frac{\kappa}{2}$ and
$\Vert g \Vert_{1}=\Vert f \Vert_{1}=\frac{\kappa}{2}$. We have

$$
\begin{aligned}
&\left\Vert\mathbb{Q}(u) - \mathbb{Q}(v) \right\Vert_{1}=
\left\Vert\mathbb{Q}(u\wedge v + f) - \mathbb{Q}(u\wedge v +g) \right\Vert_{1}
\\
&=
\left\Vert2\mathbf{Q}(u\wedge v,f) + \mathbb{Q}(u \wedge v) + \mathbb{Q}(f) - 2\mathbf{Q}(u \wedge v,g)
-\mathbb{Q}(u \wedge v) - \mathbb{Q}(g)\right \Vert_{1}
\\
&\leq
2\left\Vert\mathbf{Q}(u\wedge v,f) - \mathbf{Q}(u \wedge v,g)\right\Vert_{1}
+ \left\Vert\mathbb{Q}(f) \right\Vert_{1} + \left\Vert\mathbb{Q}(g) \right\Vert_{1}
\\
&= 2\frac{\kappa}{2}\left(1-\frac{\kappa}{2}\right)
\left\Vert\mathbf{Q}(\frac{u\wedge v}{1-\kappa/2},\frac{f}{\kappa/2})-
\mathbf{Q}(\frac{u\wedge v}{1-\kappa/2},\frac{g}{\kappa/2}) \right\Vert_{1} + \frac{\kappa^{2}}{2}
\\ &\leq \kappa\left(1-\frac{\kappa}{2}\right)2\varepsilon + \frac{\kappa^{2}}{2}
< 2\kappa \varepsilon +\frac{\kappa^{2}}{2}
\\
&= \left(2\varepsilon + \frac{\kappa}{2}\right)\kappa = \left(2\varepsilon+\frac{\kappa}{2}\right)\left\Vert u-v \right\Vert_{1}
\le \left(2\varepsilon + \varepsilon\right)\left\Vert u-v \right\Vert_{1} = 3\varepsilon \left\Vert u-v \right\Vert_{1} \ .
\end{aligned}
$$
It follows that $\mathbb{Q}$ is a strict contraction. Applying Banach's
fixed point theorem we obtain

$$
\begin{aligned}
\sup\limits_{u,v \in \uD} \left\Vert\mathbb{Q}^{n}(u)- \mathbb{Q}^{n}(v)\right\Vert_{1} \le 2 (3\varepsilon)^{n-1} \rightarrow 0,
\end{aligned}
$$
which gives that $\mathbf{Q}$ is also norm mixing. Hence, 
$\mathrm{Ball}(\mathbf{Q}^\diamond, \tfrac{1}{3}) \subseteq Int(\mathfrak{Q}_{nm})$.}
\end{exam}

\begin{cor}
$Int(\mathfrak{Q}_{nm}) \neq \emptyset$.
\end{cor}

Hence neither norm mixing nor non norm mixing quadratic stochastic operators can be considered as generic.
We introduce another type of asymptotic behaviour of quadratic stochastic operators and show that it is
prevalent in $\mathfrak{Q}$.

\begin{defn}
We say that $\mathbf{Q}\in \mathfrak{Q}$ is norm quasi--mixing if

$$
\lim_{n\to\infty} \sup\limits_{f, g, h\in \mathcal{D}} \left\Vert P_{f}^{[0,n]}(g) - P_{f}^{[0,n]}(h) \right\Vert_1 = 0 \ .
$$
The set of all norm quasi--mixing quadratic stochastic operators will
be denoted by $\mathfrak{Q}_{nqm} $.
\end{defn}
Let us note that the norm quasi--mixing condition is equivalent to

$$
\lim_{n\to\infty} \sup\limits_{f, g \in \mathcal{D}} \left\Vert P_{f}^{[0,n]}(g) - \mathbb{Q}^n(f) \right\Vert_1 = 0  .
$$
Indeed, to see the necessity of the above condition it is enough to substitute $h = f$
as $P_{f}^{[0,n]}(f) = \mathbb{Q}^n(f)$. The sufficiency follows directly
from the triangle inequality, namely

$$
\sup\limits_{f,g,h \in \mathcal{D}} \left\Vert P_{f}^{[0,n]}\left(g\right) - P_{f}^{[0,n]}\left(h\right) \right\Vert_1
\leq \sup\limits_{f,g\in \mathcal{D}} \left\Vert P_{f}^{[0,n]}\left(g\right) - \mathbb{Q}^n\left(f\right) \right\Vert_1
+ \sup\limits_{f,h \in \mathcal{D}} \left\Vert P_{f}^{[0,n]}\left(h\right) - \mathbb{Q}^n\left(f\right)\right\Vert_1 .
$$

Clearly norm mixing quadratic stochastic operators are also norm quasi--mixing, since according to the Theorem~\ref{thm6}
we have

$$
\sup\limits_{g,h\in \mathcal{D}}  \left\Vert P_{h}^{[0,n]}\left(g\right) - \mathbb{Q}^n\left(h\right) \right\Vert_1 \leq \sup\limits_{g,h\in \mathcal{D}}  \left\Vert P_{h}^{[0,n]}\left(g\right) - f \right\Vert_1 + \sup\limits_{h\in \mathcal{D}}  \left\Vert \mathbb{Q}^n\left(h\right) - f \right\Vert_1 \xrightarrow {n \to \infty} 0 \ .
$$
On the other hand, norm--quasi mixing does not imply strong (and hence norm) mixing. This can be seen in the example below.

\begin{exam}
Given a (homogeneous)
Markov operator \mbox{
$P \colon L^1 \to L^1$} let us define $\opQ \in \mathfrak{Q}$
by

$$\opQ(f,g) = \frac{1}{2}\left( \left(\int_X g d\mu \right)P f + \left(\int_X f d\mu \right)P g \right)$$ for any $f, g \in L^1$.
Then for a fixed density $f \in \uD$
and any $g \in \uD$ we have $P_f(g)  = \opQ(f,g) = \tfrac{1}{2}(Pf + Pg)$.
Thus using the fact that $P$ is a
strict contraction we get

$$\left\Vert P_f(g) - P_f(h) \right\Vert_1  = \frac{1}{2} \left\Vert  Pg - Ph \right\Vert_1 \leq \frac{1}{2} \left\Vert g -h \right\Vert_1 \ .$$
Similarly, for any natural $n$ we have

$$
\begin{aligned}
\left\Vert P_f^{[0,n]}(g) - P_f^{[0,n]}(h)\right\Vert_1 &= \left\Vert\frac{1}{2}\left(P(P_f^{[0,n-1]}(g)) -
P(P_f^{[0,n-1]}(h))\right)\right\Vert_1
\\ &
\leq \frac{1}{2} \left\Vert P_f^{[0,n-1]}(g) - P_f^{[0,n-1]}(h) \right\Vert_1 \leq \ldots \leq \frac{1}{2^n}
\left\Vert f - h \right\Vert_1 \ .
\end{aligned}
$$
Thus

$$ \sup\limits_{f,g,h\in\uD} \left\Vert P_f^{[0,n]}(g) - P_f^{[0,n]}(h)\right\Vert_1 \leq \frac{1}{2^{n-1}} \xrightarrow {n \to \infty} 0 $$
and hence $\opQ$ is norm quasi--mixing.
On the other hand, let us assume that the Markov operator $P$ does not possess any invariant density.
Since $\mathbb{Q}(f) = Pf$ for any $f \in \uD$, then neither $\mathbb{Q}$ has invariant densities.
In particular $\opQ$ is not strong (and hence not norm) mixing.
\end{exam}

Below we present our main result. We show that $\mathfrak{Q}_{nqm}$ is a large set in the topology
induced by metric $d_{u}$ and hence norm quasi--mixing can be considered as a generic property
in the class of quadratic stochastic operators.

\begin{thm} 
$\mathfrak{Q}_{nqm}$ is a dense and open subset for the metric $d_{u}$.
\end{thm}

\begin{proof}
We first show that $\mathfrak{Q}_{nqm}$ is a dense subset of $\mathfrak{Q}$ for the metric $d_{u}$.
Let $\mathbf{Q} \in \mathfrak{Q}$ be taken arbitrarily. 
For any $f, g \in L^{1}$ and any fixed density $v \in \mathcal{D}$ define, 
similarly as in Ex. \ref{exQnm},
$\mathbf{Q}^{\Diamond} \in \mathfrak{Q}$ by

$$ \mathbf{Q}^{\Diamond}(f,g):= \left(\int_{X} f \ud \mu \int_{X} g \ud \mu\right) v \ .$$
For any $0<\varepsilon<1$ define $\mathbf{Q}_{\varepsilon} \in \mathfrak{Q}$ by

$$\mathbf{Q}_{\varepsilon} = (1-\varepsilon)\mathbf{Q} + \varepsilon \mathbf{Q}^{\Diamond} \ .$$
Then

$$
\begin{array}{rcl}
d_{u}(\mathbf{Q},\mathbf{Q}_{\varepsilon}) \le \hat{d}_{u}(\mathbf{Q},\mathbf{Q}_{\varepsilon})
&=&
\sup\limits_{f,g \in \mathcal{D}} \left\Vert \mathbf{Q}(f,g) -\mathbf{Q}_{\varepsilon}(f,g) \right\Vert_{1}
\\& =&
\sup\limits_{f,g \in \mathcal{D}} \left\Vert \mathbf{Q}(f,g) -(1-\varepsilon)\mathbf{Q}(f,g) -\varepsilon\mathbf{Q}^{\Diamond}(f,g) \right\Vert_{1}
\le 2\varepsilon \ .
\end{array}
$$
Moreover, $\mathbf{Q}_{\varepsilon}$ is norm quasi--mixing. 
Indeed, consider nonhomogeneous Markov chain 
$_{\varepsilon}\opP_{f} = (_{\varepsilon}P_{f}^{[n,n+1]})_{n \geq 0}$ associated with 
$\opQ_{\varepsilon}$ and $f \in \mathcal{D}$. For any $g, h \in \mathcal{D}$ we have

$$
\begin{array}{ll}
& \left\Vert _{\varepsilon}P_{f}^{[0,n]}(g) - _{\varepsilon}P_{f}^{[0,n]}(h) \right\Vert_{1} \\
 = & \left\Vert
\mathbf{Q}_{\varepsilon}(\mathbb{Q}_{\varepsilon}^{n-1}(f),_{\varepsilon}P_{f}^{[0,n-1]}(g)) - 
\mathbf{Q}_{\varepsilon}(\mathbb{Q}_{\varepsilon}^{n-1}(f),_{\varepsilon}P_{f}^{[0,n-1]}(h)) \right\Vert_{1} \\
 = &\left\Vert
(1-\varepsilon)\mathbf{Q}(\mathbb{Q}_{\varepsilon}^{n-1}(f),_{\varepsilon}P_{f}^{[0,n-1]}(g)) + \varepsilon v
- (1-\varepsilon)\mathbf{Q}(\mathbb{Q}_{\varepsilon}^{n-1}(f),_{\varepsilon}P_{f}^{[0,n-1]}(h)) - \varepsilon v \right\Vert_{1}\\
 = & (1-\varepsilon)\left\Vert
\mathbf{Q}(\mathbb{Q}_{\varepsilon}^{n-1}(f),_{\varepsilon}P_{f}^{[0,n-1]}(g))
- \mathbf{Q}(\mathbb{Q}_{\varepsilon}^{n-1}(f),_{\varepsilon}P_{f}^{[0,n-1]}(h)) \right\Vert_{1}
\\  = &
(1-\varepsilon)\left\Vert \mathbf{Q}(\mathbb{Q}_{\varepsilon}^{n-1}(f),_{\varepsilon}P_{f}^{[0,n-1]}(g)-
_{\varepsilon}P_{f}^{[0,n-1]}(h))\right\Vert_{1}
\\ \le &
(1-\varepsilon)\left\Vert _{\varepsilon}P_{f}^{[0,n-1]}(g) -  _{\varepsilon}P_{f}^{[0,n-1]}(h) \right\Vert_{1}
\le \ldots \le 2(1-\varepsilon)^{n}
\end{array}
$$
and hence

$$
\sup\limits_{g,h \in \mathcal{D}} \left \Vert _{\varepsilon}P_{f}^{[0,n]}(g) - _{\varepsilon}P_{f}^{[0,n]}(h) \right\Vert_{1}  \le 2(1-\varepsilon)^{n} \xrightarrow {n \to \infty} 0 \ .
$$
Therefore $\mathfrak{Q}_{nqm}$ is dense in $\mathfrak{Q}$.

We now show that $\mathfrak{Q}_{nqm}$ is an open set for the metric $d_{u}$.
We first notice that for a nonhomogeneous Markov chain $\opP_{f} = (P_{f}^{[n,n+1]})_{n \geq 0}$
associated with $\opQ \in \mathfrak{Q}$ and $f \in \mathcal{D}$ the sequence

$$\sup\limits_{f,g,h \in \mathcal{D}} \left\Vert P_{f}^{[0,n]}(g) - P_{f}^{[0,n]}(h) \right\Vert_{1}$$
is non--increasing in $n$ as

$$
\begin{array}{rcl}
\sup\limits_{f,g,h \in \mathcal{D}} \left\Vert P_{f}^{[0,n+1]}(g) - P_{f}^{[0,n+1]}(h) \right\Vert_{1}
& = &
\sup\limits_{f,g,h \in \mathcal{D}} \left\Vert P_{f}^{[n,n+1]}(P_{f}^{[0,n]}(g) - P_{f}^{[0,n]}(h)) \right\Vert_{1}
 \\ & \le &
\sup\limits_{f,g,h \in \mathcal{D}} \left\Vert P_{f}^{[0,n]}(g) - P_{f}^{[0,n]}(h) \right\Vert_{1}.
\end{array}
$$
Hence $\opQ \in \mathfrak{Q}$ is norm quasi--mixing if and only if the limit condition is satisfied on some subsequence.
We will show that

$$
\mathfrak{Q}_{nqm}  = \bigcup_{n=1}^{\infty}
\left\{
\mathbf{Q} \in \mathfrak{Q}:
\sup\limits_{f,g,h \in \mathcal{D}} \left\Vert P_{f}^{[0,n]}(g) - P_{f}^{[0,n]}(h)\right\Vert_{1} < 2 \right\} \ .
$$
The inclusion

$$
\mathfrak{Q}_{nqm}  \subseteq \bigcup_{n=1}^{\infty}
\left\{
\mathbf{Q} \in \mathfrak{Q}:
\sup\limits_{f,g,h \in \mathcal{D}} \left\Vert P_{f}^{[0,n]}(g) - P_{f}^{[0,n]}(h)\right\Vert_{1} < 2 \right\}
$$
is obvious. In order to prove the opposite inclusion let us assume that there exists a natural $n$
such that $\opQ \in \mathfrak{Q}$ satisfies

$$ \sup\limits_{f,g,h \in \mathcal{D}} \left\Vert P_{f}^{[0,n]}(g) - P_{f}^{[0,n]}(h) \right\Vert_{1} < 2-\varepsilon $$
for arbitrarily small $\varepsilon > 0$.

Using the fact that $L^{1}$ is an AL--space,
the above inequality can be written in an equivalent form

$$ \inf\limits_{f,g,h \in \mathcal{D}} \left\Vert P_{f}^{[0,n]}(g) \wedge P_{f}^{[0,n]}(h)\right\Vert_{1} > \frac{\varepsilon}{2}.  $$
We introduce the following notation (for $P_{f}^{[0,k]}(g) \neq P_{f}^{[0,k]}(h)$)

$$
\begin{array}{rcl}
\gimel(k,f,g,h) & := &
\frac{P_{f}^{[0,k]}(g) - P_{f}^{[0,k]}(g) \wedge P_{f}^{[0,k]}(h)}{1 - \left\Vert P_{f}^{[0,k]}(g) \wedge P_{f}^{[0,k]}(h) \right\Vert_{1}}
- \frac{P_{f}^{[0,k]}(h) - P_{f}^{[0,k]}(g) \wedge P_{f}^{[0,k]}(h)}{1 - \left\Vert P_{f}^{[0,k]}(g) \wedge P_{f}^{[0,k]}(h) \right\Vert_{1}}
\\ & =  &\frac{P_{f}^{[0,k]}(g) - P_{f}^{[0,k]}(h)}{1 - \left\Vert P_{f}^{[0,k]}(g) \wedge P_{f}^{[0,k]}(h) \right\Vert_{1}}
\in \mathcal{D} - \mathcal{D}.
\end{array}
$$
For any natural $j$ we have

$$
\begin{array}{rcl}
\left\Vert P_{f}^{[0,jn]}(g) - P_{f}^{[0,jn]}(h) \right\Vert_{1} & = &
\left\Vert P_{\mathbb{Q}^{n}(f)}^{[0,n(j-1)]}(P_{f}^{[0,n]}(g) - P_{f}^{[0,n]}(h)) \right\Vert_{1}
\\ &= &
\left\Vert P_{\mathbb{Q}^{n}(f)}^{[0,n(j-1)]}\left(1-\left\Vert P_{f}^{[0,n]}(g) \wedge P_{f}^{[0,n]}(h) \right\Vert_{1} \right) \gimel(n,f,g,h)\right\Vert_{1}
\\ & \le & (1-\frac{\varepsilon}{2})\left\Vert P_{\mathbb{Q}^{n}(f)}^{[0,n(j-1)]}( \gimel(n,f,g,h) )\right\Vert_{1} \le \cdots \le 2\left(1-\frac{\varepsilon}{2}\right)^{j-1} \ .
\end{array}
$$
Thus

$$
\sup\limits_{f,g,h \in \mathcal{D}} \left\Vert P_{f}^{[0,jn]}(g) - P_{f}^{[0,jn]}(h) \right\Vert_{1}
\le 2\left(1-\frac{\varepsilon}{2}\right)^{j-1} \xrightarrow {j \to \infty} 0
$$
and hence $\mathbf{Q} \in \mathfrak{Q}_{nqm}$. Now let us notice that for any fixed natural $n$ the function

$$
\mathfrak{Q} \ni \opQ \mapsto \sup\limits_{f,g,h \in \mathcal{D}} \left\Vert P_{f}^{[0,n]}(g) - P_{f}^{[0,n]}(h)\right\Vert_{1}
$$
is continuous in the metric $d_{u}$. Indeed, let $f, g \in \mathcal{D}$ be fixed. Consider the family of functions

$$
\mathcal{F}_{f,g} := \left\{ F_{f,g} : (\mathfrak{Q}, d_{u}) \ni \opQ \mapsto \opQ(f,g) \in (L^{1}, \left\Vert \cdot \right\Vert_{1}) \right\} \ .
$$
Let $\varepsilon > 0$. Choose $\delta = \tfrac{\varepsilon}{4}$. For any $\opQ_{1} \in \mathfrak{Q}$ and
$F_{f,g} \in \mathcal{F}_{f,g}$, the inequality $d_{u}(\opQ_{1}, \opQ_{2}) < \delta$ implies that

$$
\begin{array}{rcl}
\left\Vert F_{f,g}(\opQ_{1}) - F_{f,g}(\opQ_{2}) \right\Vert_{1} &=& \left\Vert \opQ_{1}(f,g) - \opQ_{2}(f,g)\right\Vert_{1} \\ & \le & \hat{d}_{u}(\opQ_{1}, \opQ_{2}) \le 4 d_{u}(\opQ_{1}, \opQ_{2}) < \varepsilon \ .
\end{array}
$$
Hence $\mathcal{F}_{f,g} $ is equicontinuous. Since

$$
\begin{array}{rcl}
& & \left| \ \left\Vert \opQ_{1}(f,g) - \opQ_{1}(f,h)\right\Vert_{1} - \left\Vert \opQ_{2}(f,g) - \opQ_{2}(f,h)\right\Vert_{1} \ \right| \\ & \le & \left\Vert \opQ_{1}(f,g) - \opQ_{2}(f,g)\right\Vert_{1} + \left\Vert \opQ_{1}(f,h) - \opQ_{2}(f,h)\right\Vert_{1}
\end{array}
$$
for any $f,g,h \in \mathcal{D}$ and $\textrm{diam}(\mathcal{D}) < \infty$ then the function

$$
\mathfrak{Q} \ni \opQ \mapsto \sup\limits_{f,g,h \in \mathcal{D}} \left\Vert \opQ(f,g) - \opQ(f,h) \right\Vert_{1} = \sup\limits_{f,g,h \in \mathcal{D}} \left\Vert F_{f,g}(\opQ) - F_{f,h}(\opQ) \right\Vert_{1}
$$
is well defined and continuous in the metric $d_{u}$. In particular, the function

$$
\mathfrak{Q} \ni \opQ \mapsto \sup\limits_{f,g,h \in \mathcal{D}} \left\Vert \opQ(\mathbb{Q}^{n-1}(f), P_{f}^{[0,n-1]}(g)) - \opQ(\mathbb{Q}^{n-1}(f), P_{f}^{[0,n-1]}(h)) \right\Vert_{1}
$$
is continuous in the metric $d_{u}$. Thus we showed that for any fixed natural $n$ the set

$$ \left\{\mathbf{Q} \in \mathfrak{Q}: \sup\limits_{f,g,h} \left\Vert P_{f}^{[0,n]}(g) - P_{f}^{[0,n]}(h) \right\Vert_{1} < 2 \right\} $$
is an open subset of $\mathfrak{Q}$ in the metric $d_{u}$.

\end{proof}

It is worth emphasizing that even though $\mathfrak{Q}_{nqm}$ is a large set, not every quadratic stochastic operator
is norm quasi--mixing. This is seen in the following example.

\begin{exam}\label{ex8}
Let us consider a partition $X:=B_{1} \cup B_{2}$, $B_{1} \cap B_{2} = \emptyset$.
Define the operator $\mathbf{Q} \in \mathfrak{Q}$ for any $f, g \in L^{1}$ and some fixed $h \in L^{1}$, $h \geq 0$, by

$$
\begin{array}{rcl}
\mathbf{Q}(f,g)
& := &
\frac{h \mathbf{1}_{B_{1}}}{\int_{B_{1}} h \ud \mu } \left(
\int_{B_{1}}f \ud \mu \int_{B_{1}}g\ud \mu +
\int_{B_{1}} f \ud \mu \int_{B_{2}}g\ud \mu +
\int_{B_{2}} f \ud \mu \int_{B_{1}}g\ud \mu
\right) \\
& + & \frac{h \mathbf{1}_{B_{2}}}{\int_{B_{2}} h \ud \mu }
\int_{B_{2}}f \ud \mu \int_{B_{2}}g \ud \mu.
\end{array}
$$
Denote by $\supp{f}:= \{ x \in X: f(x) \neq 0~\mu~a.e.\}$ the support of $f \in L^{1}$. For $f \in \mathcal{D}$ 
such that $\supp{f} \subseteq B_{2}$ we have

$$ \mathbf{Q}(f,g) = \frac{h \mathbf{1}_{B_{1}}}{\int_{B_{1}} h \ud \mu } \int_{B_{1}} g \ud \mu
+ \frac{h \mathbf{1}_{B_{2}}}{\int_{B_{2}} h \ud \mu } \int_{B_{2}} g \ud \mu \ .
$$
Thus if $f, g \in L^{1}$ satisfy $\supp{f} \subseteq B_{2}$ and $\supp{g} \subseteq B_{1}$
then $P_{f}^{[0,n]}(f)=  \frac{h \mathbf{1}_{B_{2}}}{\int_{B_{2}} h \ud \mu }$ and  $P_{f}^{[0,n]}(g) =  \frac{h \mathbf{1}_{B_{1}}}{\int_{B_{1}} h\ud \mu }$.
Hence

$$
\left\Vert P_{f}^{[0,n]}(f) - P_{f}^{[0,n]}(g) \right\Vert_{1} =2.
$$
It follows that $\opQ$ is not norm quasi--mixing.
\end{exam}

\section*{Acknowledgments}
We would like to acknowledge Wojciech Bartoszek for many helpful comments and insights.
Krzysztof Bartoszek was supported by
Svenska Institutets \"Ostersj\"osamarbete scholarship nrs. 00507/2012, 11142/2013, 19826/2014.

\bibliographystyle{plainnat}
\bibliography{KBartoszekPulka}

\end{document}